\documentclass{amsart}
\usepackage{graphicx}
\usepackage{amssymb}
\usepackage{amsfonts}
\usepackage{hyperref}
\usepackage{mathrsfs}
\swapnumbers
\newtheorem{theorem}{Theorem}[section]
\newtheorem{lemma}[theorem]{Lemma}
\newtheorem{corollary}[theorem]{Corollary}
\newtheorem{proposition}[theorem]{Proposition}
\theoremstyle{definition}

\newtheorem{assumption}[theorem]{Assumption}
\newtheorem{remark}[theorem]{Remark}

\numberwithin{equation}{section}
\theoremstyle{plain}

\numberwithin{equation}{section} 
\numberwithin{figure}{section} 
\theoremstyle{plain}
\theoremstyle{plain}
\theoremstyle{remark}
\newtheorem*{acknowledgement*}{Acknowledgement}
\theoremstyle{example}

\newcommand{\cF}{{\mathcal F}}
\newcommand{\cG}{{\mathcal G}}
\newcommand{\cH}{{\mathcal H}}

\newcommand{\cX}{{\mathcal X}}
\newcommand{\cY}{{\mathcal Y}}
\newcommand{\cZ}{{\mathcal Z}}
\newcommand{\te}{{\theta}}

\newcommand{\Om}{{\Omega}}
\newcommand{\om}{{\omega}}
\newcommand{\ve}{{\varepsilon}}
\newcommand{\del}{{\delta}}

\newcommand{\sig}{{\sigma}}

\newcommand{\ka}{{\kappa}}
\newcommand{\la}{{\lambda}}

\newcommand{\bbC}{{\mathbb C}}
\newcommand{\bbE}{{\mathbb E}}

\newcommand{\bbN}{{\mathbb N}}
\newcommand{\bbP}{{\mathbb P}}
\newcommand{\bbR}{{\mathbb R}}

\newcommand{\bbZ}{{\mathbb Z}}

\begin{document}
\title[]{Limit theorems for inhomogeneous $\phi$-mixing Markov chains\
}  
 \vskip 0.1cm
 \author{Yeor Hafouta and Brendan Williams}
\address{
Department of Mathematics, The University of Florida}
\email{yeor.hafouta@mail.huji.ac.il, brendan.williams@ufl.edu }%

\thanks{ }
\dedicatory{  }
 \date{\today}

\maketitle
\markboth{Y. Hafouta Brendan Williams}{ } 
\renewcommand{\theequation}{\arabic{section}.\arabic{equation}}
\pagenumbering{arabic}

\begin{abstract}
We prove limit theorems for inhomogeneous $\phi$-mixing Markov chains. In the case of uniform $\phi$-mixing, we will focus on the purely sequential (inhomogeneous) setting, while in the non-uniform case we will focus on Markov chains in random dynamical environments. A large part of the paper is devoted to verifying our main results for Markov chains satisfying a lower Doeblin condition and for Markov chains satisfying a Dobrushin-type contraction condition. We then apply our results in the random environment case to prove estimates on random mixing times as well as limit theorems for Markovian skew products. In fact, we show that one can estimate the $\alpha$-mixing coefficients of the coordinates of the skew product, which opens the door to a variety of limit theorems. Compared with \cite{dolgopyat2023berry}, we are able to prove optimal CLT rates for inhomogeneous Markov chains solely under $\phi$-mixing, without any ellipticity assumptions. Compared with existing results for Markov chains in random ergodic environments, under mixing conditions on the environment, we are able to treat non-uniformly $\phi$-mixing chains, including non-uniform Doeblin/Dobrushin conditions, while all existing results either hold for uniformly mixing chains or are formulated under hard-to-verify conditions.
\end{abstract}



\section{Introduction}
An important discovery made in the last century is that weakly dependent random variables could satisfy the central limit theorem (CLT). 
The CLT, together with convergence rates, was obtained for various classes of stationary Markov chains (see \cite{HH}). 
However, many systems appearing in nature are non-stationary due to an interaction with the outside world.

The first result for non-stationary  Markov chains is due to Dobrushin \cite{Dub56}, where he proved that appropriately normalized partial sums $\sum_{j=1}^nY_j$ of sufficiently well mixing (contracting) Markov chains converge in distribution to the standard normal law (namely, they obey the CLT). We refer to \cite{PelCLT, SV} for a modern presentation and strengthening of Dobrushin's CLT. Since then, the central limit theorem in the inhomogeneous case has been studied extensively for uniformly bounded functionals of uniformly elliptic (in an appropriate sense)  Markov chains; see \cite{dolgopyat2023berry, dolgopyat2025berry, dolgopyat2024local, DS, merlevede2021local, merlevede2022local} for local CLT and optimal CLT rates.  See also references therein for many other related results for non-stationary dynamical systems, which by a standard time reversal argument can be reduced to Markov chains whose transition operators are the appropriately normalized transfer operators.
In this paper, we will concentrate on inhomogeneous chains which are not elliptic. More precisely, we will focus on chains satisfying a Doeblin minorization condition. 
We revisit the results in \cite{dolgopyat2023berry} and \cite{dolgopyat2025berry} and prove Berry--Esseen theorems for uniformly bounded random variables of the form $Y_j=f_j(X_j)$, where $X_j$ is an inhomogeneous $\phi$-mixing Markov chain. In particular, we show that the optimal CLT rates in \cite{dolgopyat2025berry} hold without uniform ellipticity, assuming instead only  $\phi$-mixing. 

We then study in detail two examples: Markov chains satisfying a uniform version of the Doeblin minorization condition \eqref{Doeb2} (i.e., the lower bound from the usual two-sided Doeblin condition/ellipticity), and Markov chains which are uniformly contracting in the sense of Dobrushin.
 Our approach also yields CLT with rates under some kind of sufficiently moderate, non-uniform sequential minorization condition. For instance, we can require that the Doeblin lower bound or the contraction will occur at time $j$ only after $n_j$ steps, where $n_j\to\infty$ slowly enough for the arguments to apply. 
Since the resulting conditions are not very aesthetic,  we decided not to formulate such results directly, although we provide explicit estimates in Corollaries \ref{Cor} and \ref{Cor2}, and so the readers who are interested could easily use these results together with the main results in \cite{HafSPA}, for instance.

One notable inhomogeneous example is Markov chains in random dynamical environments, which is a particular case of a random dynamical system.
Random transformations emerge in a natural way as a model
for the description of a physical system whose evolution mechanism depends on
time in a stationary way. This leads to the study of actions of compositions
of different maps or Markov operators chosen at random from a typical sequence of transformations. To fix the notation, we have an underlying probability space
$(\Om,\cF,\bbP)$ and a probability preserving map $\te:\Om\to\Om$ so that the $n$-th  Markov operator is given by 
$$
P_{\om,n}:=P_{\om}\circ P_{\te\om}\circ\cdots\circ P_{\te^{n-1}\om}.
$$
A dynamical version of this setup was discussed already in Ulam and Von-Neumann \cite{UN} and in Kakutani \cite{Ka} in connection with random ergodic theorems. 
Since then, the ergodic theory of random dynamical systems has attracted a lot of attention. See, for instance \cite{Arnold98, Cong97, Crauel2002, Kifer86, KiferLiu, LiuQian95}. We refer to the introduction of \cite[Chapter 5]{KiferLiu} for a historical discussion and applications to, for instance, statistical physics, economics, and meteorology.  

The central limit theorem and related results for Markov chains in random dynamical environments were studied in \cite{Cogburn,HK, kifer1996perron, kifer1998limit} and references therein. Similar results were obtained for wide classes of random hyperbolic systems. The research on random hyperbolic maps has exploded in recent years, and so we make no attempt in providing even a partial list. In this paper, we will use the aforementioned general sequential approach to prove limit theorems for Markov chains in random mixing dynamical environments that satisfy a certain type of non-uniformly $\phi$-mixing condition. We then study Markov chains in random environments satisfying either a non-uniform version of the lower Doeblin condition or a non-uniform random Dobrushin-type contraction. Here, the non-uniformity means that the time it takes to get the lower bound/contraction and the size of the lower bound/contraction are both random.  In applications, this is done by imposing some quantitative mixing conditions on the random noise $(\Om,\cF,\bbP,\te)$. 

Compared with \cite{Cogburn, kifer1996perron, kifer1998limit}, we are able to provide explicit conditions for the CLT, CLT rates, and large deviations in the non-uniform case. In the above applications, the main novelty in this paper is that we obtain effective polynomial or stretched exponential random ``geometric" ergodicity. Once this is achieved, the rest of the proof in the random case follows from the abstract results of \cite{YH1} (with minor notational modifications which are left for the reader).
The effective mixing rates also have applications to related problems, such as (average) mixing times and corresponding results for the skew product process, as will be discussed at the end of the paper. In examples, by assuming a lower Doeblin condition or a Dobrushin contraction, we are able to directly obtain results in total variation (or just direct estimates on the $\phi$-mixing coefficients), consequently avoiding many of the techniques and methods necessary in more general settings. These conditions are also often easier to verify than those for other quantitative limit theorems, which typically require regularity assumptions such as H\"older continuity or bounded distortion.

Note that the version of the  Doeblin condition that we use in this paper is stronger than the classical small-set ($C$-set) definition in Harris chain theory (see \cite{Brad}). In contrast, we assume that the Doeblin-type lower bound conditions hold \textit{uniformly}  for all states $x$. It would be interesting to consider the results of this paper under more general Doeblin minorization conditions.

\section{Main Result for Uniformly Mixing Inhomogeneous Markov Chains}
Let $\{X_j\}_{j\geq 0}$ be a Markov chain taking values at measurable spaces $\cX_j$. Denote 
$\cF_{k,\ell}=\sigma\{X_j: k\leq j\leq \ell, j\in\bbR\}$.
Recall that the $n$-th $\phi$-mixing coefficient of the chain is given by 
$$
\phi(n)=\sup_{j}\sup \left\{ \left|\mathbb P(B\,|\,A)-\mathbb P(B)\right|: A\in\mathcal F_{-\infty,j}, B\in\mathcal F_{j+n,\infty}, \mathbb P(A)>0\right\}.
$$
Let us take measurable functions $f_j:\cX_j\to\mathbb R, j\geq 0$ such that $\sup_j\|f_j(X_j)\|_{L^\infty}<\infty$. Denote 
$$
S_n=\sum_{j=0}^{n-1}f_j(X_j)
$$
and let $\sig_n=\sqrt{\text{Var}(S_n)}$.

\begin{theorem}\label{VarThm}
Suppose that $\phi(n)=O(\delta^n)$ for some $\delta\in(0,1)$.
The following conditions are equivalent: 
\vskip0.1cm
(i) $\sig_n$ is bounded;
\vskip0.1cm
(ii) $\liminf_{n\to\infty}\sig_n<\infty;$
\vskip0.1cm
(iii) we have
 $$
f_j(X_j)=\bbE\left[f_j(X_j)\right]+h_j(X_j)-h _{j-1}(X_{j-1})+M_j(X_{j-1},X_j),\, j>0
 $$
 for some measurable uniformly bounded functions $h_j$, and where $M_j$ is a uniformly bounded martingale difference such that $\sum_{j=1}^\infty\text{Var}\,(M_j(X_{j-1},X_j))<\infty$ (and so $\sum_{j=1}^\infty M_j(X_{j-1},X_j)$ converges almost surely and in $L^p$ for all $1\leq p<\infty$).
\end{theorem}
Next, denote 
$$
F_n(t)=\bbP(S_n-\bbE\left[S_n\right]\leq t\sig_n).
$$
\begin{theorem}\label{BE}
Suppose that $\phi(n)=O(\delta^n)$ for some $\delta\in(0,1)$ and that $\sig_n\to\infty$. Then: 

\vskip0.2cm
(i)  for every $1\leq s<\infty$, there is a constant $C_s$ such that for all $n\in\bbZ^+$,
$$
\sup_{t\in\bbR} \left(1+|t|^s \right)\left|F_n(t)-\Phi(t)\right|\leq C_s\sig_n^{-1}.
$$ 
\vskip0.1cm
(ii) for all $q>0$, we have 
$
\left\|F_n-\Phi\right\|_{L^q(dt)}=O(\sig_n^{-1}).
$
\vskip0.1cm
(iii)  for all  $1\leq s<\infty$,  there is a constant $C_s$ such that
for every absolutely continuous function $h:\bbR\to\bbR$  such that 
$H_s(h):=\int \frac{|h'(x)|}{1+|x|^s}dx<\infty$ and for every $n\in\bbN$, we have 
$$
\left|\bbE\left[h((S_n-\bbE[S_n])/\sig_n)\right]-\int h\, d\Phi\right|\leq C_s H_s(h)\sig_n^{-1}.
$$
\end{theorem}

Next, recall that the $p$-th Wasserstein distance between two probability measures $\mu,\nu$ on $\bbR$ with finite absolute moments of order $p$ is given by 
$$
W_p(\mu,\nu)=\inf_{(X,Y)\in\mathcal C(\mu,\nu)}\left\|X-Y\right\|_{L^p},
$$
where $\mathcal C(\mu,\nu)$ is the class of all  pairs of random variables $(X,Y)$ on $\bbR^2$ such that $X$ is distributed according to $\mu$, and $Y$ is distributed according to $\nu$. 

\begin{theorem}\label{ThWass}
Suppose that $\phi(n)=O(\delta^n)$ for some $\delta\in(0,1)$, and that $\sig_n\to\infty$. 
Then, for every $p<\infty$ we have 
$$
W_p (dF_n, d\Phi)=O (\sig_n^{-1}),
$$
where $dG$ denotes the measure induced by a distribution function $G$.
\end{theorem}

\subsection{Example 1: Doeblin Chains}\label{Intro1}
\begin{lemma}\label{LLLL1}
Let the    
  following uniform version of the Doeblin minorization condition hold: there exists $n_0\in\bbZ^+$, probability measures $m_j$ on $\cX_j$, and $\gamma\in(0,1)$ such that for every $j\geq0$, $x\in\cX_j$, and every measurable set $\Gamma\subseteq\cX_{j+n_0}$ we have
\begin{equation}\label{Doeb2}
P_{j,n_0}(x,\Gamma):=\bbP(X_{j+n_0}\in\Gamma\,|\,X_j=x)\geq\gamma m_{j+n_0}(\Gamma).    
\end{equation}
Then 
$$
\phi(n)\leq 2\left(1-\gamma\right)^{-n_0+1} \left(1-\gamma\right)^{n/n_0}.
$$
 \end{lemma}
\begin{proof}
Clearly, by replacing $P_{i}$ with $P_{i,n_0}$, we can just assume that $n_0=1$. In this case, let us write
 $$
 P_i(x,\cdot) = \gamma_i m_i(\cdot) + (1-\gamma)Q_i(x,\cdot)
$$ 
for some positive transition measures  $Q_i(x,\cdot)$ such that $Q_i(x,\cX_i)\leq1$. Next, let us denote by $P_{j,n}(x,\cdot)$ the transition probability of $X_{j+n}$ given $X_j=x$. Namely, as operators, we have 
$$
P_{j,n}=P_j\circ P_{j+1}\circ\cdots \circ P_{j+n-1}.
$$
We define $Q_{j,n}$ similarly.

Now, we claim that for every $j\in\bbZ$, any measurable set $\Gamma\subseteq \mathcal{X}_j$, for every $x\in\cX_{j-n}$ we have
   \begin{equation}\label{claim}
   P_{j-n,n}(x,\Gamma)=A_{j,n}(\Gamma)+(1-\gamma)^nQ_{j-n,n}(x,\Gamma),    
   \end{equation}
   where $A_{j,n}(\Gamma)$ are positive measures such that $A_{j,1}(\Gamma)=\gamma m_j(\Gamma)$ and $A_{j,n}(\Gamma)\leq A_{j,n+1}(\Gamma)$.
   In particular,
   $$
\sup_{\Gamma}\sup_x\left|P_{j-n,n}(x,\Gamma)-A_{j,n}(\Gamma)\right|\leq \left(1-\gamma \right)^n.
   $$
   Therefore,  there exists a probability measure $\mu_j$ on $\cX_j$ such that 
    $$
\sup_{\Gamma}\sup_x\left|P_{j-n,n}(x,\Gamma)-\mu_j(\Gamma)\right|\leq (1-\gamma)^n.
   $$

Now let us prove \eqref{claim}.
By relabeling the indices, it is enough to prove the lemma when $j=0$. Let us prove it by induction on $n$. For $n=1$, the lemma is equivalent to the Doeblin minorization condition. Next, suppose that the statement of the lemma holds with a specific $n$. Then
$$
P_{-n-1,n+1}(x,\Gamma)=\int P_{-n-1}(x,dy)P_{-n,n}(y,\Gamma)
$$
$$
=\int \left((\gamma m_{-n-1}(dy)+(1-\gamma)Q_{-n-1}(x,dy))\,(A_{0,n}(\Gamma)
+\gamma(1-\gamma)^n Q_{-n,n}(y,\Gamma)\right)
$$
$$
=\gamma_{-n-1}A_{0,n}(\Gamma)+\gamma(1-\gamma)^n\int Q_{-n-1}(y,\Gamma)\,dm_{-n-1}(y)
$$
$$
+(1-\gamma_{-n-1})A_{0,n}(\Gamma)+(1-\gamma)^{n+1}\int Q_{-n-1}(x,dy)\,Q_{-n,n}(y,\Gamma)
$$
$$
=A_{n}(\Gamma)+C_n(\Gamma)+(1-\gamma)^{n+1}Q_{-n-1,n+1}(x,\Gamma)
$$
for some $C_n(\Gamma)\geq 0$. 

Now, notice that 
 for a transition probability $Q$, a bounded measurable function $g$, and a probability measure $\nu$, we have 
 \begin{equation}\label{lll}
 \left|\int g(y)Q(x,dy)-\int g(y)d\nu(y)\right|\leq 2\sup|g|\sup_{\Gamma} \left|Q(x,\Gamma)-\nu(\Gamma)\right|    
 \end{equation}
 where $\Gamma$ ranges over all the underlying measurable sets. In the derivation of \eqref{lll}, we used that the total variation distance between two probability measures $\ka_1$ and $\kappa_2$ is given by 
 $$
\left\|\kappa_1-\kappa_2\right\|_{\text{TV}}=2\sup_{\Gamma}|\kappa_1(\Gamma)-\kappa_2(\Gamma)|
 $$
 and that, in general,
 $$
\left\|\kappa_1-\kappa_2\right\|_{\text{TV}}=\sup_{\left\|g\right\|_\infty\leq 1}\left|\int g\,d\kappa_1-\int g\,d\kappa_2\right|.
 $$
Therefore, by \eqref{claim} we get that
\begin{equation}\label{lll1}
\left\|P_{j,n}(G)-\mu_{j+n}(G)\right\|_{\infty}\leq 2 \left(1-\gamma \right)^n\left\|G\right\|_{\infty},    
\end{equation}
where $\|G\|_\infty=\sup|G|$.

Next, by  \cite[Ch.4]{Brad}, for every two sub-$\sig$-algebras $\cH,\cG$ of a given $\sig$-algebra,
 \begin{equation}\label{Phi rep}
 \phi(\cH,\cG)=\sup\left\{|\bbP(B\,|\,A)-\bbP(B)|: A\in\cG, B\in\cH, \bbP(A)>0\right\}    
 \end{equation}
$$
=\frac12\sup\left\{\left\|\bbE\left[g\,|\,\cG\right]-\bbE\left[g\right]\right\|_{L^\infty}:g\in L^\infty(\cH): \left\|g\right \|_{L^\infty}\leq1 \right\}. 
 $$
Now, let us take a function on $\cX_{j+n}$ such that $\|G(X_{j+1})\|_{L^\infty}\leq 1$. By possibly reducing $\cX_{j+n}$, we can assume that $|G|\leq 1$ (this will not change the left-hand side in the formula below).  Then
$$
\left\|\mathbb E[G(X_{j+j}) \,|\,X_j]-\mathbb E[G(X_{j+n})]\right\|_{L^\infty}\leq \left\|P_{j,n}G-\mu_{j+n}(G)\right\|_{L^\infty}+ \left|\mu_{j+n}(G)-\mathbb E[G(X_{j+n})]\right|.
$$
Now, notice that 
$$
\mu_{j+n}(G)-\mathbb E[G(X_{j+n})]=\mu_{j+n}(G)-\mathbb E[P_{j,n}G(X_{j})],
$$
and so 
$$
\left|\mu_{j+n}(G)-\mathbb E[G(X_{j+n})]\right|\leq \left\|P_{j,n}G-\mu_{j+n}(G)\right\|_{L^\infty}.
$$
We thus conclude from \eqref{lll1} that 
$$
\left\|\mathbb E[G(X_{j+j}) \,|\, X_j]-\mathbb E[G(X_{j+n})]\right\|_{L^\infty}\leq 4(a-\gamma)^n.
$$
Using the Markov property and then \eqref{Phi rep}, we conclude that 
$$
\phi(n)\leq 2\left(1-\gamma\right)^n,
$$
and the proof of the theorem is complete.
\end{proof}
\begin{remark}
It is a natural question what the role of the measures $\mu_j$ constructed in the proof of Lemma \ref{LLLL1} is, beyond their part in the proof that the chain is $\phi$-mixing.    
In the non-stationary setting, the measures $\mu_j$ are the analogue of the invariant measure, and are sequentially invariant. In stationary settings, for each $j$, $\mu_j$ is equal to the unique invariant measure of the Markov chain. Importantly, $\sup_j\sup_x\|P_{j,n}(x,\cdot)-\mu_{j+n}\|_{\text{TV}}\to 0$. Consequently, the measure $\mu_j$ has the following properties. Take a Markov chain $(\zeta_k)_{k\leq j}$ with transition operators $P_k$. Then the law of $\zeta_j$ given $\zeta_{j-n}$ converges in total variation towards $\mu_j$. One also gets that if we have a Markov chain $(\zeta_k)_{k\geq j}$ with transition operators $P_k$, then, regardless of the initial distribution of $\zeta_j$, the law of $\zeta_{j+n}$ has the form $\mu_{j+n}+a_n$, for some $a_n$ such that $\lim_{n\to\infty}\|a_n\|_{\text{TV}}=0$.
\end{remark}
\

 \subsection{Example 2: Dobrushin Contracting Chains}\label{Intro21}
Define the Dobrushin coefficient $\delta(P_{j,n})$ by $$\delta(P_{j,n}):=\sup_{x,y}\sup_{\Gamma}\left|\delta_x P_{j,n}(\Gamma) - \delta_y P_{j,n}(\Gamma)\right|,$$
where $\delta_x$ denotes the Dirac measure. Suppose that there exists $m\in \mathbb Z^+$  such that
\begin{align}
r:=\sup_j \delta(P_{j,m})< 1.
\end{align}

\begin{lemma}\label{DubLem}
For Markov chains satisfying the above uniform contraction, we have 
$$
\phi(n)= O(r^{n/m}).
$$
\end{lemma}
\begin{proof}
 Since $\delta(PQ)\leq \delta(P)\delta(Q)$, this means that $\delta(P_{j,n+k})\leq \delta(P_{j,n})\delta(P_{j+n,k})$ for all $j,n,k$.  This implies that for each $k\in \mathbb Z^+$, \begin{align}
    \delta(P_{j,km}) \leq r^k.
\end{align} Furthermore, for any $n\in \mathbb Z^+$, we can decompose $n = km+\ell$ for some $k\in \mathbb Z^+$ and $0\leq \ell \leq m-1$. Thus, this ensures that there exist $C>0$ and $\rho < 1$ such that for all $n$, $$\delta(P_{j,n})\leq C\rho^n.$$
Moreover, for each $x$ and $y$, for each $j,n,$ and $f\in L^{\infty}$, this gives, 
\begin{align}
    \left|P_{j,n} f(x) - P_{j,n} f(y)\right|&=\left|\int f d(P_{j,n}(x, \cdot) - P_{j,n}(y,\cdot))\right|\nonumber\\
    &\leq C\rho^n \left\|f \right\|_{L^{\infty}}.\nonumber
\end{align}
Consequently, with $\kappa_j$ denoting the law of $X_j$, we have,
\begin{align}
    |P_{j,n} f(x) - \kappa_{j+n}(f)|  &= |P_{j,n} f(x) - \int P_{j,n}f(y) \kappa_j(dy) |\nonumber\\
    &\leq \int \left |P_{j,n} f(x) - P_{j,n} f(y)\right |\kappa_j(dy)\nonumber\\
    &\leq C\rho^n \left\|f\right\|_{L^{\infty}}\nonumber.
\end{align}
Taking supremums, we obtain \begin{align}
    \left\|P_{j,n} f - \kappa_{j+n}f\right\|_{L^{\infty}} \leq  C\rho^n \left\|f\right\|_{L^{\infty}}\nonumber. 
\end{align}
It follows that the Markov chain is exponentially $\phi$-mixing using \eqref{Phi rep}.   
\end{proof}

\section{Proof of Theorems \ref{VarThm}, \ref{BE}, and \ref{ThWass}}
First, using \eqref{Phi rep} there exists $A>0$ and $\delta\in(0,1)$ such that if $\nu_j$ denotes the law of $X_{j+n}$, then
\begin{equation}\label{Phi rep 2}
\left\|P_{j,n}(G)-\nu_{j+n}(G)\right\|_{L^\infty}\leq A\left\|G\right\|_{L^\infty(\nu_{j+n})}\delta^n.    
\end{equation}
\subsubsection{Proof of Theorem \ref{VarThm}}
 Theorem \ref{VarThm} follows from the following result.
  \begin{lemma}\label{Mart lemm}
Make the same assumptions as in Theorem \ref{VarThm}. There are measurable functions $M_j$ and $h_j$ on $\cX_j$ such that almost surely for all integers $j$, we have
\begin{equation}
\label{MartCob}
f_j(X_j)-\bbE\left[f_j(X_j)\right]=M_{j}(X_{j-1},X_{j})+h_{j-1}(X_{j-1})-h_j(X_j).
\end{equation}
Moreover, $\sup_j\|h_j\|_{L^\infty}<\infty$, $\sup_j\|M_j\|_{L^\infty}<\infty$, and $M_j(X_{j-1},X_{j})$  is a  martingale difference with respect to the filtration $\cF_{-\infty,j}$.
\end{lemma}

\begin{proof}
Denote 
$\tilde f_j(X_j)=f_j(X_j)-\bbE[f_j(X_j)]$.
Set 
\begin{equation}
\label{DefUj}
h_j=\sum_{k=j+1}^{\infty}\bbE\,[\tilde f_{k}(X_k)\,|\,X_j]=\sum_{k=j+1}^{\infty}\bbE\,[\tilde f_{k}(X_k)\,|\,X_0,X_1,\cdots,X_{j}].
\end{equation}
Then by \eqref{Phi rep 2},
$$
\left\|h_j\right\|_{L^\infty}\leq A\sum_{k=1}^\infty \delta^{k}\left\|f_{j+k}\right\|_{L^\infty}\leq C'<\infty
$$
for some constants $C,C'>0$.
 Set $M_{j}=\tilde f_j(X_j)+h_j(X_j)-h_{j-1}(X_{j-1})$. Then, it follows from the definition of $h_j$ that $M_j$ is a martingale difference with respect to the filtration $\cF_{0,j}$.
\end{proof}

\subsection{Proof of Theorems \ref{BE} and \ref{ThWass}}

By applying  \cite[Theorem 5]{H} and \cite[Theorem 9]{H} and \cite[Corollary 11]{H}, 
Theorems \ref{BE} and \ref{ThWass} will follow from the following result.
\begin{proposition}\label{Growth Prop}
In the circumstances of  Theorems \ref{BE} and \ref{ThWass}, we can develop a branch $\Lambda_{n}(t)$ of $\ln\bbE[e^{itS_n}]$ around the origin such that for every $k\geq3$, there are constants
 $\del_k>0$ and $C_k>0$ such that for all $s\leq k$, we have
\begin{equation}\label{Add}
\sup_{t\in[-\del_k,\del_k]}|\Lambda^{(s)}_{n}(t)|\leq C_k\sig_n^2.    
\end{equation}
\end{proposition}
\begin{proof}
The proof of Proposition \ref{Growth Prop} proceeds similarly to the proof of the main result in \cite{dolgopyat2023berry}, but for the reader's convenience, we provide some of the details. 
First, since the chain is $\phi$-mixing by \cite[Theorem 6.17]{PelBook}, and because the functions $f_j$ are uniformly bounded for every $k$, there exists a constant $c_k>0$ such that for all $j$ and $n$, we have 
\begin{equation}\label{Pell}
\left\|\sum_{\ell=j}^{j+n-1}(f_j(X_j)-\bbE[f_j(X_j)])\right\|_{L^k}\leq c_k\left(1+\left\|\sum_{\ell=j}^{j+n-1}(f_j(X_j)-\bbE[f_j(X_j)])\right\|_{L^2}\right).    
\end{equation}

Next, let us decompose $\{0,1,\cdots,n-1\}$ into blocks  $B_j, j\leq k_n$, such that $R\leq \text{Var}(\sum_{\ell\in B_j}f_\ell(X_\ell))\leq 2R$ for some $R$ large enough.  Then, using that the chain is exponentially fast $\rho$-mixing  (recalling that $\rho(n)\leq 2\sqrt{\phi(n)}$), we see that 
if $R$ is large enough, then with $S_{B_j}=\sum_{\ell\in B_j}f_\ell(X_\ell)$,
$$
C_1\sum_{j=1}^{k_n}\text{Var}\left(S_{B_j}\right)-C_2=\text{Var}\left(\sum_{j\leq k_n}S_{B_j}\right)\leq C_3\sum_{j=1}^{k_n}\text{Var}\left(S_{B_j}\right)+C_4
$$
for some constants $C_i>0$. We thus conclude that $k_n\asymp \sigma_n^2$.

Next, define operators $P_{j,z}:L^\infty(\mathcal X_{j+1})\to L^\infty(\mathcal X_j), z\in\mathbb C$,  by
$$
P_{j,z}g(x)=\bbE\,[e^{zf_{j+1}(X_{j+1})}g(X_{j+1})\,|\,X_j=x].
$$
Then, these operators 
 are holomorphic in $z\in\bbC$ with respect to the $L^\infty$ norms. Now, since \eqref{Phi rep 2} holds, we can apply 
 \cite[Theorem D.2]{dolgopyat2025berry}. This means that for $z\in\bbC$ with $|z|$ small enough,  there are analytic in $z$ uniformly bounded complex numbers $\la_j(z)$, functions $h_j^{(z)}$ such that $\sup_z\sup_j\|h_j^{(z)}\|_{L^\infty(\nu_j)}<\infty$, and complex linear functionals $\nu_j^{(z)}$ such that $\sup_z\sup_j\|\nu_j^{(z)}\|_{L^\infty(\nu_j)}<\infty$, 
 $\nu_j^{(z)}(h_j^{(z)})=\nu_j^{(z)}(\textbf{1})=1$, $\lambda_j(0)=1$, $h_j^{(0)}=\textbf{1}$, $\nu_j^{(0)}=\mu_j$, and 
 $$
\left\|P_{j,z}\circ\cdots\circ P_{j+n-1,z}-\left(\prod_{k=j}^{j+n-1}\lambda_k(z)\right)\nu_{j+n}^{(z)}\otimes h_j^{(z)}\right\|_{L^\infty}\leq C\delta^n,
 $$
 where $C>0$ and $\delta\in(0,1)$ are constants.

Then, the corresponding sequential pressure functions $\pi_j(z)$ are given by $\pi_j(z)=\ln\la_j(z)$. Denote $\Pi_{j}(z)=\sum_{\ell\in B_j}\pi_\ell(z)$. Then, arguing like in \cite{dolgopyat2025berry},
 and using inequality \eqref{Pell}, it follows that for every $k\geq 3$, there exist $C_k,\delta_k>0$ such that 
 $$
\sup_{j}\sup_{t\in[-\delta_k,\delta_k]} \left|\Pi_{j}^{(k)}(it)\right|\leq C_k.
 $$
Therefore, 
$$
\sup_{t\in[-\delta_k,\delta_k]}\left|\sum_{j=1}^{k_n}\Pi_{j}^{(k)}(it)\right|\leq C_k'\sigma_n^2.
$$
This estimate, together with the above complex sequential Perron--Frobenius theorem and the arguments in the proof of \cite[Proposition 30]{dolgopyat2023berry}, yields \eqref{Add}.
\end{proof}

\section{Main Results for Markov Chains in Random Dynamical Environments}\label{Intro2}
Let $\cX$ be a measurable space and let $\cX_\om$ be measurable in $\omega$ subsets of $\cX$. Let $P_\om(x,\Gamma)$ be transition probabilities  which are measurable in $\om$ (here $x\in\cX_
\om$ and $\Gamma\subseteq\cX_{\te\om}$ is a measurable set). Denote 
$$
P_{\om,n}=P_{\om}\circ P_{\te\om}\circ\cdots\circ P_{\te^{n-1}\om}.
$$
Let us consider any Markov chain $X_\omega=(X_{\om,n})_{n\in \mathbb Z^+ }$ such that $X_{\om, n}$ takes values on $\cX_{\te^n\omega}$ and the transition operator when jumping from time $n$ to time $n+1$ is $P_{\te^n\omega}$. Let us denote by $\kappa_\om$ the law of $(X_{\om,k})_{k\in\mathbb Z}$.
Denote $\cF_{\om,k,\ell}=\sigma\{X_{\om,j}: k\leq j\leq \ell, k\in\bbR\}$ and
$$
\phi_\omega(n;k)=\sup\left\{|\kappa_\om(B\,|\,A)-\kappa_\om(B)|: A\in\cF_{\om,0,k}, B\in \cF_{\om,k+n,\infty}\right\}.
$$
We assume here that there is a random variable $K\in L^p$ for some $p>1$ such that the $\phi$-mixing coefficients of the Markov chain $X_\omega$ satisfy
\begin{equation}\label{Ass1}
\phi_\omega(n;k)\leq K(\te^k\omega)n^{-p}.    
\end{equation}
\begin{remark}
It is enough to assume that $\phi_\omega(n;0)\leq K(\omega)n^{-p}$, since $\phi_{\om}(n;k)=\phi_{\te^k\om}(n;0)$. 
\end{remark}

Let us take a function $f:\Omega\times X$ such that $\|f_\om\|_{\infty}\in L^p$. Let 
$$
S_n^\om f=\sum_{j=0}^{n-1}f_{\te^j\omega}(X_{\om,j}).
$$
\begin{theorem}\label{Variance asymptotics theorem}
If $p>2$, then
there exists $\Sigma\geq0$ such that for $\bbP$-a.a. $\om$, we have 
$$
\lim_{n\to\infty}\frac1n\text{Var}\,(S_n^\om f)=\Sigma^2.
$$
Moreover, $\Sigma=0$ if and only if $f_{\te^k\om}(X_{\om,k})-\bbE[f_{\te^k\om}(X_{\om,k})]=h_{\te^k\om}(X_{\om,k})-h_{\te^{k+1}\om}(X_{\om,k+1})$
for some measurable uniformly bounded functions $h_\om:\cX_\om\to\bbR$.
\end{theorem}

\begin{theorem}\label{CLT}
Let$\bar{S}_n^\om=S_n^\omega-\mathbb E[S_n^\omega]$. If $p>4$  and $\Sigma>0$,  then for $\bbP$-a.a. $\om$  and all  $t\in\bbR$,
 $$
\lim_{n\to \infty}\ka_\om \left\{n^{-1/2}\bar S_{n}^\om f\leq t\Sigma\right\}=\frac{1}{\sqrt{2\pi}}\int_{-\infty}^t e^{-\frac12x^2}dx:=\Phi(t).
 $$
\end{theorem}

\begin{theorem}\label{BE1}
If $p>4$  and $\Sigma>0$, then for $\bbP$-a.a. $\om$,
  $$
\sup_{t\in\mathbb R}\left|\ka_\om\{\bar S_n^\om f\leq \Sigma_{\om,n}t\}-\Phi(t)\right|=O\left(n^{-(\frac{1}2-\delta(p))} \right),
  $$
 where $\delta(p)\to 0$ as $p\to\infty$.
\end{theorem}

We can also get close to optimal moderate deviations principles.
\begin{theorem}\label{MDP}
If $p>4$ and $\Sigma>0$, then we have the following.
 Let $\{a_n\}_{n\in \mathbb Z^+}$ be a sequence such that 
$
\lim_{n\to\infty}a_n=\infty\,\,\text{ and }\,\,a_n=o(n^{\frac12-\delta(p)}). 
$
Denote $W_n^\om=\frac{\bar S_n^\om f}{n^{1/2}a_n}$. Then, for $\bbP$-a.a. $\om$,
for every Borel-measurable subset $\Gamma\subseteq\bbR$,
$$
-\inf_{x\in\Gamma^o}\frac 12x^2\leq \liminf_{n\to\infty}\frac{1}{a_n^2}\ln \ka_\om \left\{W_n^\om\in \Gamma\right\}\leq \limsup_{n\to\infty}\frac{1}{a_n^2}\ln \ka_\om \left\{W_n^\om\in \Gamma\right\}\leq -\inf_{x\in\overline\Gamma}\frac 12x^2,
$$
where $\overline{\Gamma}$ is the closure of $\Gamma$ and $\Gamma^o$ is its interior.
\end{theorem}

Using \eqref{Phi rep}, we see that if $\mu_\omega$ is the law of $X_{\omega,0}$, then
$$
\left\|P_{\omega,n}g-\mu_{\te^n\omega}(g)\right\|_{L^\infty}\leq 2K(\te^n\omega)\left\|g\right\|_{L^\infty}n^{-p}.
$$
Therefore, the theorems are proven exactly like the main results in \cite{YH1}, and we will omit the proofs. The reason we formulated them is that in the next two sections, we will provide examples where Assumption \eqref{Ass}  holds, and so it is convenient to have a general theorem to refer to.

\subsection{Example 1: Markov Chains in Random Environments Satisfying a Random Doeblin Condition}
Let $\{Y_j\}_{j\in\bbZ}$ be a stationary mixing sequence of random variables taking values on some measurable space $\cY$. Let $(\Omega,\mathcal F,\mathbb P,\te)$ be the shift system generated by this sequence, namely $\Omega=\cY^\bbZ$, $\te:\Omega\to\Omega$ is the left shift, and $\bbP$ is the law of the path $\{Y_j\}_{j\in\bbZ}$.

We assume that there are random variables $n_\om$ and $\gamma_\om\in(0,1)$ and a probability measure $m_\om$ on $\cX_\om$ which is measurable in $\omega$
such that $\bbP$-a.s. for all $x\in\cX_{\te^{-n_\om}\om}$ and every measurable subset $\Gamma\subseteq\cX_\om$, we have
\begin{equation}\label{DoebRand}
P_{\te^{-n_{\om}}\omega,n_\om}(x,\Gamma)\geq \gamma_\om m_{\omega}(\Gamma).    
\end{equation}
Then for every $n\geq n_\omega$, 
$$
P_{\te^{-n}\om,n}(x,\Gamma)\geq \gamma_\om m_{\omega}(\Gamma).
$$

Next, denote by $\cF_{a,b}$ the $\sig$-algebra generated by $Y_s$ for all finite $a\leq s\leq b$.
Recall that the $\alpha$-mixing coefficients of $\{Y_j\}$ are given by 
$$
\alpha(n)=\sup_{k}\left\{|\bbP(A\cap B)-\bbP(A)\bbP(B)|: A\in\cF_{-\infty,k}, \,B\in\cF_{k+n,\infty}\right\},
$$
and the $\psi_U$-mixing (upper $\psi$-mixing) coefficients of $\{Y_j\}$ are given by 
$$
\psi_U(n)=\sup_{k}\left\{\frac{\bbP(A\cap B)}{\bbP(A)\bbP(B)}-1:\, A\in\cF_{-\infty,k}, \,B\in\cF_{k+n,\infty},\, \bbP(A)\bbP(B)>0\right\}.
$$
\begin{assumption}\label{Ass}
For every $\delta>0$ and  $M\in\bbN$, set
$$
A=A_{\delta,M}= \left\{\om: \gamma_\om\geq \delta,\, n_\om\leq M\right\}.
$$
We assume that there exist $\del,M$ such that $\bbP(A)>0$ and the following condition holds.
For every $r$, there exist sets $A_r\in\cF_{-r,r}$ and $B_r\in\cF$ such that 
$$
\beta_r=\bbP(B_r)=o(r^{-1}), \,\,A_r\subseteq A\cup B_r,\,\, p_0=\lim_{r\to\infty}\bbP(A_r)>0.
$$
 Moreover, suppose that either 
$\alpha(r)=O(r\beta_r)$ or 
\begin{equation}\label{Psii}
 (1+\limsup_{k\to\infty}\psi_U(k))\,(1-p_0(1-\delta))<1.   
\end{equation}
\end{assumption}
Notice that when $P_\omega$ and $\gamma_\om$ depend only on $\om_0$, then $A\subseteq\cF_{-M,M}$, and so we can take $B_r=\emptyset$ and $A_r=A$ (and $\delta$ small enough and $M$ large enough). Observe further that when $n_\omega$ is minimal and  $P_\omega$ and $\gamma_\om$ depend only on $\om_{-L}, \cdots, \om_L$ for a fixed $L \in \mathbb{Z}^+$, then $A\in \mathcal F_{-r,r}$ for all $r \geq L$. Thus, we may still take $A_r = A$ and $B_r = \emptyset$ (where $\delta$ is small enough and $M$ large enough so that $\mathbb P(A)>0$.)

\begin{proposition}\label{LL}
Under Assumption \ref{Ass}, for $\mathbb P$-a.a. $\om$, there exists a probability measure $\mu_\om$ on $\mathcal X_\omega$, and make one of the two assumptions:
\vskip0.1cm
(i) Suppose   $b_r=O(r^{-a})$ for some $a>3$. Let $\beta$ and $d$ be such that $a>3+\beta d$. Then, there exists $K\in L^d$ such that
$$
\sup_{x\in\cX_{\te^{-n}\om}}\sup_{\Gamma\subseteq\cX_\om}\left|P_{\te^{-n}\om,n}(x,\Gamma)-\mu_\om(\Gamma)\right|\leq K(\om)n^{-\beta}.
$$
Therefore, if we take the chains $X_{\om,k}$ such that $X_{\om,k}$ is distributed according to $\mu_{\te^k\omega}$, then for every $\epsilon>0$,
$$
\phi_\omega(n;k)\leq  A_\epsilon(\te^k\om)n^{-\beta+\epsilon},
$$
where $A_\epsilon\in L^d$.
\vskip0.1cm
(ii) Suppose $b_r=O(e^{-ar^\eta})$ for some $a,\eta>0$. Then, for every $0<\zeta<\frac{\eta}{1+\eta}$, there exists a random variable $K\in\cap_{q\geq1} L^q$ such that 
$$
\sup_{x\in\cX_{\te^{-n}\om}}\sup_{\Gamma\subseteq\cX_\om}\left|P_{\te^{-n}\om,n}(x,\Gamma)-\mu_\om(\Gamma)\right|\leq K(\om)e^{-n^\zeta}.
$$
Therefore, if we take the chains $X_{\om,k}$ such that $X_{\om,k}$ is distributed according to $\mu_{\te^k\omega}$, then
$$
\phi_\omega(n;k)\leq  A(\te^k\om)e^{-n^\zeta},
$$
where $A\in L^q$ for all $q$.
\end{proposition}
\begin{remark}
Clearly $(P_\omega)^*\mu_\om=\mu_{\te\om}$, for $\bbP$-a.a. $\om$. Note also that when $P_\omega$ and $\gamma_\om$ depend only on $\om_0$, we can take $\eta$ arbitrarily large, since $\bbP(B_r)=0$ for $r$ large enough. Thus, we can take $\zeta$ arbitrarily close to $1$ (assuming \eqref{Psii} or $\alpha(r)=O(c^r)$ for some $c\in(0,1)$).
\end{remark}

The proof of Proposition \ref{LL} is based on the following result, which we formulate in the sequential setting. That is, we assume that we have a sequence of measurable spaces $\cX_j$ and transition probabilities $P_j(x,dy)$ from $\cX_j$ to $\cX_{j+1}$. As before, denote 
$$
P_{j,n}=P_j\circ P_{j+1}\circ\cdots\circ P_{j+n-1}.
$$
Let us suppose that for every $j$, there exists $\ell_{j}$ such that for every $x\in\cX_{j-\ell_j}$,
\begin{equation}\label{Doeb21}
P_{j-\ell_j,\ell_j}(x,\Gamma)\geq \gamma_{j-\ell_j}m_{j}(\Gamma).
\end{equation}
Then, for every $n\geq\ell_j$ and $x\in\cX_{j-n}$ we have 
\begin{equation}\label{Doeb21.1}
P_{j-n,n}(x,\Gamma)\geq \gamma_{j-\ell_j}m_{j}(\Gamma).
\end{equation}
In the setup of Section \ref{Intro1}, we can take $\ell_j=n_0$ and $\gamma_j=\gamma\in(0,1)$.
 Define $N_{j,1}=\ell_{j-1}$,  $N_{j,2}=\ell_{j-1-\ell_{j-1}}$, and then by recursion $N_{j,n}=\ell_{j-1-N_{j,n-1}}$. Using these notions and Lemma \ref{LLLL1}, the following result follows.
\begin{corollary}\label{Cor}
There are increasing positive measures $A_{j,n}$  such that 
for all $j$ and $m$ and $n>N_{j,m-1}$, we have
\begin{equation}
\sup_{\Gamma}\sup_x\left|P_{j-n,n}(x,\Gamma)-A_{j,n}(\Gamma)\right|\leq \prod_{k=1}^{m-1}(1-\gamma_{N_{j,k}}).
\end{equation}
   Therefore, if $\prod_{k=1}^{m-1}(1-\gamma_{N_{j,k}})\to 0$ as $m\to\infty$, then there exists a probability measure $\mu_j$ on $\cX_j$ such that 
    $$
\sup_{\Gamma}\sup_x\left|P_{j-n,n}(x,\Gamma)-\mu_j(\Gamma)\right|\leq \prod_{k=1}^{m-1}\left(1-\gamma_{N_{j,k}}\right).
   $$
   Consequently, by \eqref{Phi rep}, if we take a Markov chain $(X_j)$ with transition probabilities $P_j$ then with 
   $\cF_{k,\ell}=\sigma\{X_s: k\leq s\leq \ell, s\in\bbR\}$, we have
   $$
   \phi(\cF_{-\infty,j-n},\cF_{j,\infty})\leq 2\prod_{k=1}^{m-1}(1-\gamma_{N_{j,k}}).
   $$
\end{corollary}

\begin{corollary}\label{Cor2}
In the above notations,
$$
\left\|P_{j-n,n}g-\mu_{j}(g)\right\|_{L^\infty}\leq \left\|g(X_j)\right\|_{L^\infty} \prod_{k=1}^{m-1}(1-\gamma_{N_{j,k}}).
$$
In particular, in the setting of Section \ref{Intro1}, we have 
$$
\left\|P_{j-n,n}g-\mu_{j}(g)\right\|_{L^\infty}\leq \left(1-\gamma \right)^{-n_0}\left\|g(X_j)\right\|_{L^\infty} \left(1-\gamma\right)^{n/n_0}.
$$
\begin{proof}
By \cite[Ch.4]{Brad}, for every two sub-$\sig$-algebras $\cH,\cG$ of a given $\sig$-algebra,
 $$
\phi(\cH,\cG)=\sup\left\{|\bbP(B\,|\,A)-\bbP(B)|: A\in\cG, B\in\cH, \bbP(A)>0\right\}
$$
$$
=\frac12\sup\left\{ \left\|\bbE\left[g\,|\,\cG\right]-\bbE\left[g\right]\right\|_{L^\infty}:g\in L^\infty(\cH): \left\|g\right\|_{L^\infty}\leq1 \right\}. 
 $$
\end{proof}
\end{corollary}

Now, let us return to the random setting.

\subsubsection*{\textbf{Proof of Proposition \ref{LL}}.}  By \cite[Lemma 4.9]{YH} and under Assumption \ref{Ass}, we have: 
\begin{lemma}\label{L1}
(i) If   $b_r=O(r^{-a})$, let $\beta>0$ and $d\geq 1$ be such that $a>3+\beta d$. Then, there exists $K\in L^d$ such that 
$$
(1-\delta)^{\sum_{j=1}^{[n/M]-1}\mathbb{I}({\te^{-jM}\om\in A})}\leq K(\om)n^{-\beta}.
$$
Consequently,
$$
\sup_{\Gamma}\sup_x\left|P_{\te^{-n}\om,n}(x,\Gamma)-\mu_\om(\Gamma)\right|\leq K(\om)n^{-\beta}.
$$
\vskip0.1cm
(ii) If $b_r=O(e^{-ar^\eta})$, then for every $0<\zeta<\frac{\eta}{1+\eta}$,  there exists a random variable $K\in\cap_{q\geq1} L^q$ such that 
$$
(1-\delta)^{\sum_{j=1}^{[n/M]-1}\mathbb{I}{(\te^{-jM}\om\in A)}}\leq K(\om)e^{-n^\zeta}.
$$
Consequently,
$$
\sup_{\Gamma}\sup_x\left|P_{\te^{-n}\om,n}(x,\Gamma)-\mu_\om(\Gamma)\right|\leq K(\om)e^{-n^\zeta}.
$$
\end{lemma}
Note that in the case when $P_\omega$ and $\gamma_0$ depend only on $\om_0$, we can take $\zeta$ arbitrarily close to $1$.
\begin{corollary}
(i) In the circumstances of  Lemma \ref{L1} (i), for every $\varepsilon>1/p$ we have  
$$
\sup_{\Gamma}\sup_x\left|P_{\om,n}(x,\Gamma)-\mu_{\te^n\om}(\Gamma)\right|\leq \tilde K_\varepsilon(\om)\,n^{-\beta+\varepsilon},
$$
where $\tilde K_\varepsilon\in L^p$.
\vskip0.1cm
(ii) In the circumstances of  Lemma \ref{L1} (ii), for every $q>1$ and $0<\zeta<\frac{\eta}{1+\eta}$, there exists $\tilde K=\tilde K_{q,\zeta}\in L^q$ such that 
$$
\sup_{\Gamma}\sup_x\left|P_{\om,n}(x,\Gamma)-\mu_{\te^n\om}(\Gamma)\right|\leq \tilde K(\om)e^{-n^\zeta}.
$$
\end{corollary}
 \begin{proof}
 Let $\tilde K(\om)=\sup_{n\geq 1}(n^{-\varepsilon}K(\te^n\om))$. Then, for every $q$ such that $K\in L^q$,
 $$
\bbE [(\tilde K(\om))^q]\leq \left\|K \right\|_{L^q}^q\sum_{n}n^{-q\varepsilon}<\infty
 $$
 if $\varepsilon>1/q$.
 \end{proof}

Now the proof of Proposition \ref{LL} is completed using that,
as before, by \cite[Ch.4]{Brad}, for any two sub-$\sig$-algebras $\cH,\cG$ of a given $\sig$-algebra,
 $$
\phi(\cH,\cG)=\sup\left\{|\bbP(B\,|\,A)-\bbP(B)|: A\in\cG, B\in\cH, \bbP(A)>0\right\}
$$
$$
=\frac12\sup\left\{\left\|\bbE\left[g\,|\,\cG\right]-\bbE\left[g\right]\right\|_{L^\infty}:g\in L^q(\cH): \left \|g \right\|_{L^\infty}\leq1\right\}. 
 $$
\qed



\subsection{Example 2: Markov Chains in Random Environments Satisfying a Random Dobrushin-type Contraction}
Henceforth, we assume that there exist random variables $n_\omega$ and $\delta_\omega<1$ such that 
$$
\delta(P_{\om,n_\omega})\leq \delta_\omega.
$$
Let us take $r<1$ and $M\in\mathbb N$ such that $\bbP(A)>0$, where $A=\{\om: n_\omega\leq M, \delta_\omega<r\}$.
\begin{lemma}\label{DubLem1}
We have 
$$
\phi_\omega(n;0)\leq r^{\prod_{j=1}^{[n/M]-1}\mathbb I(\theta^{jM}\omega\in A)}.
$$
\end{lemma}
\begin{proof}
The proof is similar to the proof of  Lemma \ref{DubLem}, but since certain modifications are needed and the proof is short, we present all the details. 
Since $\delta(PQ)\leq \delta(P)\delta(Q)$, this means that $\delta(P_{\omega,n+k})\leq \delta(P_{\om,n})\delta(P_{\te^n\om,k})$ for all $j,n,k$.  Using also that  $\delta(Q)\leq 1$, this implies that for each $k\in \mathbb Z^+$, 
\begin{align}
    \delta(P_{\om,k}) \leq \prod_{j=0}^{[k/M]-1}\delta(P_{\te^{jM}\om,M})\leq  r^{\prod_{j=1}^{[n/M]-1}\mathbb I(\theta^{jM}\omega\in A)}.
\end{align} 
Moreover, for each $x$ and $y$, for each $j,n,$ and $f\in L^{\infty}$, this gives, 
\begin{align}
    |P_{\om,n} f(x) - P_{\om,n} f(y)|&=\left|\int f d(P_{\om,n}(x, \cdot) - P_{\om,n}(y,\cdot))\right|\nonumber\\
    &\leq r^{\prod_{j=1}^{[n/M]-1}\mathbb I(\theta^{jM}\omega\in A)}\left\|f\right\|_{L^{\infty}}.\nonumber
\end{align}
This, with $\nu_\om$ denoting the law of $X_{\om}$, then gives 
\begin{align}
    \left|P_{\om,n} f(x) - \nu_{\te^n\om}(f)\right|  &= |P_{\om,n} f(x) - \int P_{\om,n}f(y) \nu_\om(dy) |\nonumber\\
    &\leq \int \left|P_{\om,n} f(x) - P_{\om,n} f(y) \right|\,\nu_\om(dy)\nonumber\\
    &\leq r^{\prod_{j=1}^{[n/M]-1}\mathbb I(\theta^{jM}\omega\in A)} \left\|f\right\|_{L^{\infty}}\nonumber.
\end{align}
Taking supremums, we obtain \begin{align}
    \left\|P_{\om,n} f - \nu_{\te^n\om}f\right\|_{L^{\infty}} \leq  r^{\prod_{j=1}^{[n/M]-1}\mathbb I(\theta^{jM}\omega\in A)}\left\|f\right\|_{L^{\infty}}\nonumber.
\end{align}
It follows that the Markov chain is exponentially $\phi$-mixing using \eqref{Phi rep}.       
\end{proof}

Next, take $\beta(\omega)=r^{\mathbb I(\omega\in A)}$. Then $\mathbb E_\bbP[\beta]<1$. 
\begin{assumption}\label{Ass11}
We assume the following condition holds.
For every $r$, there exist sets $A_r\in\cF_{-r,r}$ and $B_r\in\cF$ such that 
$$
\beta_r=\bbP(B_r)=o(r^{-1}), \,\,A_r\subseteq A\cup B_r,\,\, p_0=\lim_{r\to\infty}\bbP(A_r)>0.
$$
 Moreover, suppose that either 
$\alpha(r)=O(r\beta_r)$ or 
\begin{equation}\label{Psii1}
 (1+\limsup_{k\to\infty}\psi_U(k))\,(1-p_0(1-\delta))<1.   
\end{equation}
\end{assumption}
Notice that when $P_\omega$ and $\gamma_\om$ depend only on $\om_0$, then $A\subseteq\cF_{-M,M}$, and so we can take $B_r=\emptyset$ and $A_r=A$ (and $\delta$ small enough and $M$ large enough). Observe further that when $n_\omega$ is minimal and  $P_\omega$ and $\gamma_\om$ depend only on $\om_{-L}, \cdots, \om_L$ for a fixed $L \in \mathbb{Z}^+$, then $A\in \mathcal F_{-r,r}$ for all $r \geq L$. Thus, we may still take $A_r = A$ and $B_r = \emptyset$ (where $\delta$ is small enough and $M$ large enough so that $\mathbb P(A)>0$.)

Now,  by \cite[Lemma 4.9]{YH} and under Assumption \ref{Ass}, we have: 
\begin{lemma}\label{L2}
(i) If   $b_r=O(r^{-a})$, let $\beta>0$ and $d\geq 1$ be such that $a>3+\beta d$. Then, there exists $K\in L^d$ such that 
$$
r^{\prod_{j=1}^{[n/M]-1}\mathbb I(\theta^{jM}\omega\in A)}\leq K(\om)n^{-\beta},
$$
and so $\phi_\om(n;k)=\phi_{\te^k\om}(n;0)\leq K(\te^k\om)n^{-\beta}$.
\vskip0.1cm
(ii) If $b_r=O(e^{-ar^\eta})$, then for every $0<\zeta<\frac{\eta}{1+\eta}$,  there exists a random variable $K\in\cap_{q\geq1} L^q$ such that 
$$
r^{\prod_{j=1}^{[n/M]-1}\mathbb I(\theta^{jM}\omega\in A)}\leq K(\om)e^{-n^\zeta},
$$
and so $\phi_\om(n;k)=\phi_{\te^k\om}(n;0)\leq  K(\te^k\om)e^{-n^\zeta}$.
\end{lemma}

\subsection{Application to statistics of the random mixing times}\label{Mixing times}
Given $\ve>0$, we define 
$$
N_\ve(\om)=\min\left\{n\in\bbN: \left\|P_{\om,n}-\mu_{\te^n\om}\right\|_{L^\infty}\leq \ve\right\}.
$$
Let $a_n$ be either $O(e^{-an^\zeta})$ for some $a,\zeta>0$ or $O(n^{-A})$ for some $A>1$.  
\begin{assumption}\label{a}
Suppose that $\bbP$-a.s. for all $n\in\bbN$,
   $$
\left\|P_{\om,n}-\mu_{\te^n\om}\right\|_{L^\infty}\leq K(\omega) a_n,
    $$
    where $K(\om)\in L^p$.
\end{assumption}

\begin{theorem}
Under Assumption \ref{a}, for every $\ve>0$ we have 
$$
\bbP(\om: N_\ve(\om)>N)\leq \|K\|_{L^p}^p\,\ve^{-p}a_N^{p}.
$$
\end{theorem}

\begin{proof}
We have 
$$
\left\{\om: N_\ve(\om)>N\right\}=\left\{\om: \left\|P_{\om,n}-\mu_\om \right\|_{L^\infty}\geq \ve,\,\,\forall n\leq N\right\}
\subseteq\left\{\om: K(\om)a_n\geq \ve,\,\,\forall n\leq N\right\}
$$
$$
=\left\{\om: K(\omega)\geq \ve a_N^{-1}\right\}.
$$
Thus, the result follows by the Markov inequality.
\end{proof}

\subsection{Skew Products}\label{skew product section}
Let us denote $\cZ_\om=\prod_{k\in\bbZ}\cX_{\theta^k\om}\subseteq\cZ=\cX^\bbZ$. 
 Let us define the skew product sequence $Z_n(\omega,z)$ by 
$$
Z_n(\om,z)=\left(\te^n\omega,(z_{j+n})_{j
\in \bbZ}\right), \, z=(z_n)\in\cZ.
$$
 Noting that $(P_\om)^*\mu_{\te\om}=\mu_\om$, $\bbP$-a.s., we see that $\mu_\om$ is a Markov measure on $\cY_\om$, and we can consider Markov chains $X_\om:=\{X_{\om,k}\}_{k\in\bbZ}$ whose law on $\cY_\om$ is $\mu_\om$. Such chains have transition probabilities $P_{\te^k\om}$.
 
Let $\kappa_\omega$ denote the law of the Markov chain $(X_{\om,k})_k$ on $\cZ_\om$. 
Let us view $Z_n$ as a sequence of random variables 
with respect to the measure $\kappa=\int\kappa_\om d\bbP(\omega)$. Now, define the projection onto the $n$-th coordinates $\pi_n:\Omega\times\cZ\to\cY\times\cX$ by
$$
\pi_n(\om,z)=(\om_n,z_n).
$$
 Then $\pi_n=\pi_0\circ Z_n$. Let  $\alpha_\pi$ denote the $\alpha$-mixing coefficients of $(\pi_n)$ and for each fixed environment $\omega$, let $\phi_\omega(n)$ denote the
$\phi$-mixing coefficients of the stationary Markov chain
$X_\omega$.
\begin{proposition}\label{prp}
Suppose that $P_\om$ depends only on $\om_0$.
Assume also that $\phi_\om(n)\leq K(\omega)b_n$, where either $b_n=e^{-an^\zeta}, a,\zeta>0$, or $b_n=n^{-c}, c>0$ and $K\in L^p, p\geq 1$. Then,
$$
\alpha_{\pi}(n)=O(b_{[n/2]}+\alpha([n/2])),
$$
where $\alpha(\cdot)$ are the $\alpha$-mixing coefficients of $(Y_j)$.
\end{proposition}
\begin{proof}[Proof of Proposition \ref{prp}]
In the course of the proof, we will use some ideas from \cite{Skew}. Let us take two measurable functions $G,H$ such that $|G|,|H|\leq1$, and write 
$$
\mathbb E_{\kappa}\left[G(\pi_{j+n},\pi_{j+n+1}, \cdots)\,H(\cdots,\pi_{j-1},\pi_{j})\right]
$$
$$
=\int\kappa_{\omega}\left(G (\om_{j+n,\infty},(X_{\om,s})_{s\geq j+n}\right)\, H\left(\omega_{-\infty,j}, (X_{\om,k})_{k\leq j}\right)d\mathbb P(\omega),
$$
where $\om_{k,\ell}=(\om_s)_{k\leq s\leq\ell}$ (where $s\in\mathbb R$). Now, we have 
$$
\Big|\kappa_{\omega}(G(\om_{j+n,\infty},(X_{\om,s})_{s\geq j+n})\, H(\omega_{-\infty,j}, (X_{\om,k})_{k\leq j})
$$
$$
-\kappa_{\omega}(G(\om_{j+n,\infty},(X_{\om,s})_{s\geq j+n}))\,\kappa_\om(H(\omega_{-\infty,j}, (X_{\om,k})_{k\leq j})\Big|\leq K(\te^j\om)b_n.
$$
Now, since $p\geq 1$, we see that 
$$
\Big|
\mathbb E_{\kappa} \left[G(\pi_{j+n},\pi_{j+n+1}, \cdots)\, H( \cdots,\pi_{j-1},\pi_{j})\right]
$$
$$
-\int \kappa_{\omega}(G(\om_{j+n,\infty},(X_{\om,s})_{s\geq j+n}))\,\kappa_\om(H(\omega_{-\infty,j}, (X_{\om,k})_{k\leq j})d\bbP(\om)\Big|\leq Cb_n.
$$
Next, denote 
$$
\tilde G(\om)=\kappa_{\omega}\left(G(\om_{j+n,\infty},(X_{\om,s})_{s\geq j+n})\right), \,\,\,\tilde H(\omega)=\kappa_\om\left(H(\omega_{-\infty,j}, (X_{\om,k})_{k\leq j}\right).
$$
Then, since $\mu_\omega$ depends only on $\om_s, s\leq 0$ (as it is the limit of $P_{\te^{-n}\om,n}$), we see that $\tilde H$ is measurable with respect to $\cF_{-\infty,j}$. Next, notice that 
$$
\tilde G(\om)=\kappa_{\te^{j+n}\om} \left(G(\omega_{j+n,\infty},(X_{\te^{j+n}\om,s})_{s\geq 0}\right)
$$
$$
=\int \mathbb E \left[G(\omega_{j+n,\infty},(X_{\te^{j+n}\om,s})_{s\geq 0})\,|\,X_{\te^{j+n}\om}=x\right]\,d\mu_{\te^{j+n}\om}(x).
$$
Now, the function $\mathbb E[G(\omega_{j+n,\infty},(X_{\te^{j+n}\om,s})_{s\geq 0})\,|\,X_{\te^{j+n}\om}=x]$ is bounded by 1 in absolute value, and it depends only on $\om_s,s\geq j+n$, since $P_{\om}$ depends only on $\om_0$. Moreover, since 
$$
\left\|\mu_{\te^{j+n}\om}-P_{\te^{j+n-[n/2]}\om,[n/2]}\right\|_{L^\infty}\leq A(\te^{j+[n/2]}\om)\, b_{n-[n/2]}
$$
for $A\in L^p$,
and because $P_{\te^{j+n-[n/2]}\om,[n/2]}$ is measurable with respect to $\cF_{j+[n/2],\infty}$, we conclude by the minimization property of conditional expectations and since $A\in L^p$ that 
\begin{equation}\label{uuu}
\left\|\tilde G-\mathbb E[\tilde G \,|\,\cF_{j+[n/2],\infty}]\right\|_{L^p}\leq\left\|A\right\|_{L^p}\,b_{n-[n/2]}\leq Cb_{[n/2]}.   
\end{equation}
Now, recall that (see \cite{Hall}) for every $A\in L^\infty(\cF_{-\infty,j})$ and $B\in L^\infty(\cF_{j+m,\infty})$, we have 
$$
\left|\text{Cov}_{\mathbb P}(A,B)\right|\leq4\, \left\|A\right\|_{L^\infty}\left\|B\right\|_{L^\infty}\alpha(m). 
$$
Using this, \eqref{uuu}, and that $|G|,|H|\leq1$, we see that
$$
\left|\text{Cov}_{\mathbb P}(\tilde G,\tilde H)\right|\leq 2\,Cb_{[n/2]}+4\,\alpha\left([n/2]\right),
$$
and the proof of the proposition is complete.
\end{proof}

Proposition \ref{prp} applies to processes with sufficiently fast $\alpha$-mixing rates, which naturally occur in numerous processes. For example, geometrically ergodic Markov chains (such as those satisfying the Doeblin condition or the Dobrushin contraction condition of this paper), or irreducible and aperiodic chains on finite-state spaces apply here. Also, mixing subshifts of finite type with a Gibbs measure are applicable because they are known to have exponential $\psi$-mixing, which implies exponential $\alpha$-mixing. Also, Proposition \ref{prp} can be applied to the bilinear models, ARMA processes, and ARX$(k,q)$ nonlinear processes (all under the assumptions of \cite[Section 2.4]{doukhan1995mixing}).

\section{Applications of Proposition \ref{prp}}
The main benefit of Proposition \ref{prp} is that it allows us to apply $\alpha$-mixing estimates to skew product processes. In this section, we include representative applications of Proposition \ref{prp}. These applications demonstrate that Proposition \ref{prp} is useful in transferring classical and nonconventional annealed limit theorems for $\alpha$-mixing sequences (which has a well-developed literature of results) to skew product sequences obtained from Markov chains in random environments. 
\subsection{A Nonconventional Functional CLT} Applying \cite[Theorem 2.3]{kifer2014nonconventional} to the stationary process $(\pi_j)$ (and referring to the discussion after the theorem statement for $\alpha$-mixing estimates), we obtain the following nonconventional functional CLT as an application of Proposition \ref{prp}.

\begin{corollary}
    In the context of Proposition \ref{prp},
suppose Proposition \ref{prp} gives $\alpha_\pi (n) \leq C n^{-p}$ for large enough $p$.
    Let $F: (\mathcal{Y} \times \mathcal X)^\ell \to \mathbb R$ satisfy the assumptions of  \cite[Theorem 2.3]{kifer2014nonconventional},  let $F_i$ denote the functions induced from $F$ as in \cite[(2.17), (2.18)]{kifer2014nonconventional}, 
    and let  $q_i$ be functions satisfying \cite[(2.9), (2.10),  and (2.11)]{kifer2014nonconventional}.
   For $t\in [0,1]$, define 
   \begin{align*}
       \xi_{n}(t) = \sum_{i=1}^k \xi_{i,n}(it) + \sum_{i=k+1}^\ell \xi_{i,n}(t),
   \end{align*}
   where for $1\leq i \leq k$,

   \begin{align*}
       \xi_{i,n}(t) = \frac{1}{\sqrt{n}}\sum_{j=1}^{[ nt/i ]}F_i \left(\pi_j, \pi_{2j}, \cdots, \pi_{ij} \right),
   \end{align*}
       and for $i\geq k+1$,
   \begin{align*}
       \xi_{i,n}(t) = \frac{1}{\sqrt{n}} \sum_{j=1}^{[ nt ]} F_i \left(\pi_{q_1(j)}, \cdots, \pi_{q_i(j)} \right).
   \end{align*} Then, the vector-valued process
   $(\xi_{1,n},\cdots, \xi_{\ell, n})$ converges jointly in distribution as $n\to \infty$ to the Gaussian process $(\eta_1, \cdots, \eta_\ell)$ with stationary independent increments and covariances given by \cite[Proposition 4.1]{kifer2014nonconventional}. 
   In particular, $\xi_n(\cdot)$ converges to the Gaussian process $\xi(\cdot)$ given by $$\xi(t) = \sum_{i=1}^k \eta_i(it) + \sum_{i=k+1}^\ell \eta_i(t).$$
\begin{proof}
    Proposition \ref{prp} provides the required polynomial $\alpha_\pi$-mixing rates. Therefore,  the assumptions of \cite[Theorem 2.3]{kifer2014nonconventional} are satisfied, and the proof is complete.
\end{proof}
   
\end{corollary}

\subsection{A Nonconventional Moderate Deviations Principle}

We apply the cumulant estimates of \cite[Corollary 3.2]{hafouta2020nonconventional} to obtain the following nonconventional moderate deviations principle.

\begin{corollary}
    In the setting of Proposition \ref{prp}, suppose that Proposition \ref{prp} gives $\alpha_\pi(m)\leq de^{-cm^\eta}$ for all $m$ for constants $c, \eta>0$ and $d\geq 1$. Let $$S_n = \sum_{j=1}^n F \left(\pi_{q_1(j)}, \cdots, \pi_{q_\ell(j)}\right),$$ where $F$ and the $q_i$ satisfy the assumptions for the extension of \cite[Theorem 2.6]{hafouta2020nonconventional} in the setting of sub-exponentially $\alpha$-mixing sequences of bounded product functions described in \cite[Section 3.4.1]{hafouta2020nonconventional}. Then, the normalized sum satisfies the conclusions of \cite[Theorem 2.6]{hafouta2020nonconventional}.
\end{corollary}

\begin{proof}
    Proposition \ref{prp} provides the required stretched exponential $\alpha_\pi$-mixing estimates. Therefore, the assumptions of the extension of \cite[Theorem 2.6]{hafouta2020nonconventional} to the setting of $\alpha$-mixing sequences of bounded product functions described in \cite[Section 3.4.1]{hafouta2020nonconventional} are satisfied by our assumptions and the $\alpha$-mixing rates ensured by Proposition \ref{prp}. 
\end{proof}

\subsection{A Law of the Iterated Logarithm}
\begin{corollary}\label{lil corollary}
    In the setting of Proposition \ref{prp}, let  $S_n = \sum_{j=1}^n f(\pi_j)$ for a centered $f:\mathcal Y \times \mathcal X \to \mathbb R$. Suppose that the $\alpha_\pi$-mixing coefficients provided by Proposition \ref{prp} are such that the DMR condition of \cite{rio2017asymptotic} is satisfied for the observable $f$ (see \cite[Corollary 1.2]{rio2017asymptotic}). 
    Then, $$\limsup_{n\to \infty}\frac{|S_n|}{\sigma_n\sqrt{\log\log n}}\leq 8 \,\,\,\,\,\,\text{almost surely.}$$
\end{corollary}

\begin{proof}
    Because $(\pi_j)$ is stationary, by Proposition \ref{prp} and our assumptions, the assumptions of \cite[Theorem 6.4]{rio2017asymptotic} are satisfied, and the proof is immediate.
\end{proof}

\subsection{Other Applications} Note that the applications in this section are representative and not exhaustive, as a vast literature of $\alpha$-mixing results already exists. We provide a few more applications, but note that this is only a partial list. Beyond the applications in this section, we can obtain a strong law of large numbers (see \cite[Section 3.3]{rio2017asymptotic}) under suitable assumptions. We can also obtain an almost sure invariance principle (see \cite[Section II]{shao1987strong}) under stronger assumptions than Corollary \ref{lil corollary}, as an application. Proposition \ref{prp} can also be applied to obtain concentration inequalities for $\alpha$-mixing sequences (for instance, see \cite{rio2017asymptotic}).


\begin{thebibliography}{99}


\bibitem{Arnold98}
 L. Arnold, {\em Random Dynamical Systems}, Springer-Verlag, New York, Berlin (1998).
 

\bibitem{Brad}
R.C. Bradley, {\em Introduction to Strong Mixing Conditions}, Volume 1, Kendrick Press, Heber City, 2007.

\bibitem{Cong97}
 N.D. Cong, {\em Topological Dynamics of Random Dynamical Systems}, Oxford Univ. Press, Oxford (1997).
 
 \bibitem{Cogburn}
R. Cogburn, {\em On the central limit theorem for Markov chains in random environments}, Ann. Prob. 19, 587--604 (1991).



\bibitem{Crauel2002}
H. Crauel, Random Probability Measures on Polish Spaces, Taylor \& Francis, London (2002).

\bibitem{Dub56}
R. Dobrushin {\em Central limit theorems for non-stationary Markov chains I, II} Theory Probab.
Appl. 1 (1956) 65-80, 329--383.

\bibitem{dolgopyat2023berry}
D. Dolgopyat and Y. Hafouta, {\em A Berry-Esseen theorem and Edgeworth expansions for uniformly elliptic inhomogeneous Markov chains}, Probab. Theory Relat. Fields 186, 439–476 (2023).



\bibitem{dolgopyat2025berry}
D. Dolgopyat and Y. Hafouta, {\em Berry Esseen theorems for sequences of expanding maps}, Probab. Theory Related Fields (2025), in press, https://doi.org/10.1007/s00440-025-01368-7.

\bibitem{dolgopyat2024local}
D. Dolgopyat and Y. Hafouta, {\em Local limit theorems for expanding maps}, http://arxiv.org/abs/2407.08690.


\bibitem{DS}
D. Dolgopyat, O. Sarig,  {\em Local limit theorems for inhomogeneous Markov chains}, Springer Lecture Notes in Mathematics series, 2023.

\bibitem{doukhan1995mixing}
Doukhan, Paul, {\em Mixing: Properties and Examples}, Springer Lecture Notes in Statistics Vol. 85, Springer, 1995.

\bibitem{HK}
Y. Hafouta and Yu. Kifer, {\em Nonconventional limit theorems and random dynamics}, 
World Scientific, Singapore, 2018.


\bibitem{H}
Y. Hafouta, 
{\em Non-uniform Berry-Esseen theorem and Edgeworth expansions with applications to transport distances for weakly dependent random variables}, https://arxiv.org/abs/2210.07204v3 (2022).

\bibitem{HafSPA}
Y. Hafouta, {\em Convergence rates in the functional CLT for $\alpha$-mixing triangular arrays}, Stochastic Processes and their Applications 161, 242--290 (2023).

\bibitem{Skew}
Y. Hafouta, {\em Large deviations, moment estimates and almost sure invariance principles for skew products with mixing base maps and expanding-on-average fibers}, Ergodic Theory and Dynamical Systems 44.1 (2024): 118--158.


\bibitem{YH}
Y. Hafouta, {\em Effective (moderate) random RPF theorems and applications to limit theorems for non-uniformly expanding RDS}, https://arxiv.org/abs/2311.12950v3.

\bibitem{YH1}
Y. Hafouta, {\em Spectral methods for limit theorems for random expanding transformations}, https://arxiv.org/abs/2311.12950v4.

\bibitem{hafouta2020nonconventional} Y. Hafouta, {\em Nonconventional moderate deviations theorems and exponential concentration inequalities}, Ann. Inst. H. Poincar\'{e} Probab. Statist. 56(1), 428--448 (2020)

\bibitem{Hall}
P.G. Hall and  C.C. Hyde, {\em Martingale central limit theory and its application},
Academic Press, New York, 1980.

\bibitem{HH}
H. Hennion and L. Herv\'e, {\em Limit Theorems for Markov Chains and Stochastic Properties of Dynamical
Systems by Quasi-Compactness}, Lecture Notes in Mathematics vol. 1766, Springer, Berlin, 2001.

\bibitem{Ka}
S. Kakutani, {\em Random ergodic theorems and Markoff processes with a stable distribution},
Proc. 2nd Berkeley Symp. on Math. Stat. and Probab., 1951, pp. 247--261.



\bibitem{Kifer86}
Y. Kifer, {\em Ergodic Theory of Random Transformations}, Birkh\"auser, Boston (1986).


\bibitem{KiferLiu}
Y. Kifer, P.-D Liu, {\em Random Dynamics}, Handbook of dynamical systems, Cambridge, 1995.


\bibitem{kifer1996perron}
Yu. Kifer, {\em Perron-Frobenius theorem, large deviations, and random perturbations
in random environments},
 Math. Z. 222(4) (1996), 677-698.


\bibitem{kifer1998limit}
Yu. Kifer, {\em Limit theorems for random transformations and processes in random
environments},
Trans. Amer. Math. Soc. 350 (1998), 1481-1518.

\bibitem{kifer2014nonconventional} Yu. Kifer and S. R. S. Varadhan, {\em Nonconventional limit theorems in discrete and continuous time via martingales}, Ann. Probab. Vol. 42 (2014) No. 2, 649–688




\bibitem{LiuQian95} 
P.-D. Liu and M. Qian, {\em Smooth Ergodic Theory of Random Dynamical Systems}, Lecture Notes in Math.,
Vol. 1606, Springer-Verlag, Berlin (1995). 


\bibitem{liu1998random}
P.-D. Liu, {\em Random perturbations of Axiom A sets}. Jour. Stat. Phys., 90 (1998), 467--490.



\bibitem{Nag61}
S.V. Nagaev, {\em More exact statements of limit theorems for homogeneous Markov chains}, 
Theory Probab. Appl. 6 (1961), 62--81.


\bibitem{PelCLT}
M. Peligrad, {\em Central limit theorem for triangular arrays of non-homogeneous Markov chains}, Prob. Theor. Rel. Fields. 2012.

\bibitem{PelBook}
F. Merlev\'ede,  M. Peligrad, M. and S. Utev, S, {\em Functional Gaussian Approximation for Dependent Structures}, Oxford University Press (2019).

\bibitem{merlevede2021local} 
F. Merlev\`ede,  M. Peligrad, and C. Peligrad,
{\em On the local limit theorems for $\psi$-mixing Markov chains,}
 ALEA {\bf 18} (2021) 1221--1239.

 \bibitem{merlevede2022local} 
 F. Merlev\`ede,  M. Peligrad, and C. Peligrad,
{\em On the local limit theorems for lower $\psi$-mixing Markov chains,}
ALEA  {\bf 18}  (2022) 1103--1121.


\bibitem{rio2017asymptotic}
E. Rio,
{\em Asymptotic theory of weakly dependent random processes,}
 Probability Theory and Stochastic Modelling, Vol. 80, Springer (2017).

 \bibitem{shao1987strong}
 Shao, Qi-Man and LIU, Chuan-Rong,
 {\em Strong approximations for partial sums of weakly dependent random variables,}
 Sci. China Ser. A-Math 30 (1987), Vol. 6 575--587.

\bibitem{SV}
S. Sethuraman, S. R. S. Varadhan, {\em A martingale proof of Dobrushin’s theorem for nonhomogeneous Markov chains}, Electron. J. Probab. 10 (2005) paper 36, 1221–1235.


\bibitem{UN}
 S.M. Ulam and J. von Neumann, {\em Random ergodic theorems, Bull. Amer. Math. Soc.}
51 (1945), 660.

\end{thebibliography}
\end{document}